\documentclass[11pt,leqno]{amsart}
\usepackage[english]{babel}
\usepackage[T1]{fontenc} 
\usepackage[utf8]{inputenc}

\usepackage{amsmath}
\usepackage{amsthm}
\usepackage{amsfonts,amssymb,mathtools}

\usepackage{mathtools}
\usepackage[mathcal]{eucal}

\usepackage{hyperref}
\usepackage[nice]{nicefrac}

\usepackage{microtype}
\usepackage{chngcntr}
\usepackage{ifthen}

\usepackage{mathrsfs}

\usepackage{ dsfont }

\usepackage{graphicx}
\usepackage[all]{xy}

\theoremstyle{plain}

\newtheorem{thm}{Theorem}[section]
\newtheorem{lem}[thm]{Lemma}
\newtheorem{prop}[thm]{Proposition}
\newtheorem{cor}[thm]{Corollary}

\theoremstyle{remark}

\numberwithin{equation}{section}

\DeclareMathOperator{\Cl}{Cl}
\DeclareMathOperator{\Gal}{Gal}

\DeclareMathOperator{\rk}{rk}

\DeclareMathOperator{\ord}{ord}

\title[On the $16$-rank of class groups]{On the 16-rank of class groups of $\mathbb{Q}(\sqrt{-3p})$ for primes $p$ congruent to $1$ modulo $4$}

\author[M. Piccolo]{Margherita Piccolo} 

\address{Margherita Piccolo:
  Mathematisches Institut, Heinrich-Heine-Universit\"at, 40225
  D\"usseldorf, Germany} \email{margherita.piccolo@hhu.de}
  


\begin{document}

\begin{abstract}
\vspace{5pt}

For fixed $q\in\{3,7,11,19, 43,67,163\}$, we consider the density of primes $p$ congruent to $1$ modulo $4$ such that the class group of the number field $\mathbb{Q}(\sqrt{-qp})$ has order divisible by $16$. We show that this density is equal to $1/8$, in line with a more general conjecture of Gerth.
Vinogradov's method is the key analytic tool for our work.
\end{abstract}

\maketitle


\section{Introduction}

\vspace{5pt}

The study of the 2-parts of the class groups of quadratic number fields is an active area of research. We recall that, for $k\in \mathbb{N}$, the $2^k$-rank of a finite abelian group $G$ is the dimension of the $\mathbb{F}_2$-vector space $2^{k-1}G/2^k G$.
Milovic~\cite{Milovic} studied the density for the 16-rank in certain particular thin families of quadratic number fields. 
Koymans and Milovic~\cite{16-rank}~\cite{-p} proved density results for the 16-rank in families of imaginary quadratic number fields of the form $\mathbb{Q}(\sqrt{-p})$ for primes $p$ and $\mathbb{Q}(\sqrt{-2p})$ for primes $p$ congruent to 1 modulo 4. \newline
These results are in line with Gerth's conjecture~\cite{Gerth}, which extends a conjecture of Cohen and Lenstra~\cite{C-L} to include the 2-part. It is expected that the group $2\Cl(K)[2^\infty]$ satisfies the Cohen--Lenstra heuristic, where $K$ varies over imaginary quadratic number fields and $\Cl(K)[2^\infty]$ denotes the 2-part of the class group $\Cl(K)$. More recently, Smith~\cite{Smith} proved Gerth's conjecture and he gave a new powerful method to study the 2-part of class groups, but it is uncertain whether this new method is applicable to \emph{thin} families that we are about to consider.

Our aim is to continue the work of Koymans and Milovic, by proving results for the 16-rank of the class groups of \emph{thin} families of imaginary quadratic number fields.
The first natural case to consider is $\mathbb{Q}(\sqrt{-3p})$, in accordance with the title of the article.
For technical reasons, we restrict to primes $p$ congruent to 1 modulo 4, so that only two primes divide the discriminant. In this situation, we obtain that the 2-part of the class group $\Cl(\mathbb{Q}(\sqrt{-3p}))$ is non-trivial and cyclic, by Gauss genus theory.

\noindent Moreover, our approach to $\mathbb{Q}(\sqrt{-3p})$, extends to families $K_{p,q}:=\mathbb{Q}(\sqrt{-qp})$ with fixed $q \in Q:=\{3,7,11,19, 43,67,163\}$ and $p$ varying over all primes congruent to 1 modulo 4.
The elements in $Q$ are the complete list of primes $q$ congruent to 3 modulo 4 such that the field $\mathbb{Q}(\sqrt{-1}, \sqrt{q})$ is a principal ideal domain (see~\cite{Uchida}). These conditions are useful in the technical considerations of the analytic part of our work.

Our main result is the following.\newpage
\begin{thm}
	\label{Result}
	Let $q\in \{3,7,11,19,43,67,163\}$ be fixed. For primes $p$, let $h(-qp)$ denote the class number of the imaginary quadratic number field $K_{p,q}=		\mathbb{Q}(\sqrt{-qp})$. For each prime $p$ congruent to 1 modulo 4, set
	\[
	e_p=
		\begin{cases}1&\text{if }16\mid h(-qp),\\
		-1&\text{if } 8\mid h(-qp)\text{ but }16\nmid h(-qp),\\
		0&\text{otherwise.}
		\end{cases}
	\]
	Then we have 
	\[
	\sum_{\substack{p\leqslant x\\p\equiv 1 \bmod 4}}e_p\ll x^{1-\frac{1}{3200}} \quad \quad \text{for }x>0.
	\]
\end{thm}
\noindent In the theorem, $\ll$ denotes the Vinogradov symbol for $\mathit{O}(\,\,)$.

We will see in \textsection 3.3, that the numbers $e_p$ are not always zero and so our result shows that the sequence $e_p$ oscillates as $p$ varies. Indeed, if $8$ divides the class number, 16 divides it approximately half of the time.

\begin{cor} 
	\label{corollary}
	Let $q \in \{3,7,11,19,43,67,163\}$. Then the limit 
		\[
		\delta(16):=\lim_{x\to \infty} \frac{\#\{p\leqslant x:\,p\equiv 1\bmod 4, \,\, 16\mid h(-pq)\}}{\#\{p\leqslant x:\,p\equiv 1\bmod 4\}},
		\]
	exists and
	$\delta(16)=\frac{1}{8}.$ 
\end{cor}

The main tool we use is the generalized version of Vinogradov's method in the setting of number fields, given by Friedlander \emph{et al\@.}~\cite{spin of prime ideals}, similarly as in the works of Koymans and Milovic~\cite{16-rank}~\cite{-p}.
Moreover, as in~\cite{16-rank}, our results are unconditional in contrast to the work of Friedlander \emph{et al\@.}~\cite{spin of prime ideals}, which uses a conjecture on short character sums.\newline
The key ingredient of our argument is a sequence defined in \textsection 3.5, that encodes when 16 divides the class number of $K_{p,q}$. We carry out careful estimation of the so-called sums of type I and sums of type II that are needed to use Vinogradov's method.


\section{Prerequisites}


\subsection{Hilbert symbols and \texorpdfstring{$n$}{n}-th power residue symbol}

Let $K$ be a number field and let $\mathcal{O}_K$ be its ring of integers.
Let $n$ be a natural number and denote by $\mu_n$ the group of $n$-th roots of unity in $\mathbb{C}$. 
Let $K_\mathfrak{p}$ be the completion of $K$ with respect to a finite prime $\mathfrak{p}$ of $K$.
We assume that $K_\mathfrak{p}^\times$ contains a primitive $n$-th root of unity.
Then $L_\mathfrak{p}:=K_\mathfrak{p}(\sqrt[n]{K_\mathfrak{p}^\times})$ is the maximal abelian extension of exponent $n$ of $K_\mathfrak{p}$, by Kummer theory.\newline
We employ the notation of~\cite[Chapter V, \textsection 3]{Neukirch}. The $n$-th Hilbert symbol is the non-degenerate bilinear pairing
	\begin{align*}
	\left(\frac{\,\,\,,\,\,\,}{\mathfrak{p}}\right)_{K,n}:K_{\mathfrak{p}}^\times/ (K_\mathfrak{p}^{\times})^{ n}&\times K_{\mathfrak{p}}^\times/ (K_\mathfrak{p}^{\times})^{ n} \longrightarrow \mu_n\\	(a&,b)\longmapsto\frac{\sigma_a(\sqrt[n]{b})}{\sqrt[n]{b}},
	\end{align*}
where $\sigma_a$ is the corresponding element of $a$ in $\Gal(L_\mathfrak{p}/K_\mathfrak{p})$, given by the isomorphism $K_{\mathfrak{p}}^\times/ (K_\mathfrak{p}^{\times})^{ n}\cong \Gal(L_{\mathfrak{p}}/K_{\mathfrak{p}})$ of class field theory.
We recall basic properties of this symbol, see~\cite[Chapter V, \textsection 3, Proposition 3.2]{Neukirch}.

\begin{prop}
	\label{properties of Hilbert}
	For all $a$, $a'$, $b$, $b'\in K_{\mathfrak{p}}^\times/ (K_\mathfrak{p}			^{\times})^{ n}$, the $n$-th Hilbert symbol has the following properties: 
	\begin{itemize}
		\item[(i)]$\left(\frac{a a',b}{\mathfrak{p}}\right)_{K,n}=\left(\frac{a,b}{\mathfrak{p}}\right)_{K,n}\left(\frac{a',b}{\mathfrak{p}}\right)_{K,n},$
		\item[(ii)]$\left(\frac{a ,bb'}{\mathfrak{p}}\right)_{K,n}= \left(\frac{a,b}{\mathfrak{p}}\right)_{K,n}\left(\frac{a,b'}{\mathfrak{p}}\right)_{K,n},$
		\item[(iii)]$\left(\frac{a,b}{\mathfrak{p}}\right)_{K,n}=1\Leftrightarrow a$ lies in the image of the norm map of the extension $K_{\mathfrak{p}}(\sqrt[n]{b})/K_{\mathfrak{p}},$
		\item[(iv)]$\left(\frac{a,b}{\mathfrak{p}}\right)_{K,n}=\left(\frac{b,a}{\mathfrak{p}}\right)_{K,n}^{-1},$
		\item[(v)]$\left(\frac{a,1-a}{\mathfrak{p}}\right)_{K,n}=1$ and $\left(\frac{a,-a}{\mathfrak{p}}\right)_{K,n}=1,$
		\item[(vi)] if $\left(\frac{a,b}{\mathfrak{p}}\right)_{K,n}=1$ for all $b \in K_{\mathfrak{p}}^\times$, then $a\in K_{\mathfrak{p}}^{\times n}.$
	\end{itemize}
\end{prop}
Let $\mathfrak{p}$ be a finite prime of $K$ that does not divide $n$ and let $a$ be an invertible element of the valuation ring of $K_\mathfrak{p}$. Denote by $N$ the norm of the prime ideal $\mathfrak{p}$ i.e. $N:=\mathrm{N}_{K/\mathbb{Q}}(\mathfrak{p})$.
The $n$-th power residue symbol $\left(\frac{a}{\mathfrak{p}}\right)_{K,n}\in \mu_n$, is defined by the congruence
\begin{equation}
	\label{n-th power residue symbol}
	\left(\frac{a}{\mathfrak{p}}\right)_{K,n}\equiv a^{\frac{{N}-1}{n}}\bmod \mathfrak{p}.
\end{equation}
For every odd ideal $\mathfrak{b}$ of $\mathcal{O}_K$ (i.e. coprime to 2) that is coprime to $n$, and every element $a\in\mathcal{O}_K$ coprime to $\mathfrak{b}$, i.e. $\gcd((a), \mathfrak{b})=(1)$, we define the $n$-th power residue symbol by
\begin{equation}
	\label{def n-th residue 1}
	\left(\frac{a}{\mathfrak{b}}\right)_{K,n}:=\prod_{\mathfrak{p} | \mathfrak{b}} \left(\frac{a}{\mathfrak{p}}\right)_{K,n}^{\ord_\mathfrak{p} (\mathfrak{b})}
\end{equation}
and we set $\left(\frac{a}{\mathfrak{b}}\right)_{K,n}=0$ if $a$ is not coprime to $\mathfrak{b}.$

\noindent For $b\in \mathcal{O}_K$, we define
\begin{equation}
	\label{def n-th residue 2}
	\left(\frac{a}{b}\right)_{K,n}:=\left(\frac{a}{b\,\mathcal{O}_K}
	\right)_{K,n}.
\end{equation}
For $K=\mathbb{Q}$ we omit the subscript $K$.


\subsection{Quartic reciprocity}

The quadratic and the quartic residue symbols will be the ones that we will use the most. Since we will work in the field $M_q:=\mathbb{Q}(\sqrt{-1},\sqrt{q}),$ for $q\in Q$ with $Q=\{3,7,11,19,43,67,163\}$, we will state a weak version of the quartic reciprocity law in this setting. 

\begin{lem}
	\label{reciprocità quartica}
	Let $a, b \in \mathcal{O}_{M_q}$ with $b $ odd. If we fix $a$, then $
	\left(\frac{a}{b}\right)_{M_q,4}$ depends only on the congruence class of 
	$b$ modulo $32a\mathcal{O}_{M_q}.$
	Moreover, if $a$ is odd, then
 	\[
 	\left(\frac{a}{b}\right)_{M_q,4}=\mu\cdot\left(\frac{b}{a}\right)_{M_q,4},
 	\]
	where $\mu \in \{\pm 1 ,\pm i\}$ depends only on the congruence classes of 
	$a$ and $b $ modulo $32\mathcal{O}_{M_q}.$
\end{lem}

\begin{proof}
First, let us focus on the second part of the lemma and fix $a\in\mathcal{O}_{M_q}$.
If $a$ and $b$ are not coprime to each other, then on both sides of the identity we have 0. Now, suppose that they are coprime to each other and that $q\neq 7$. Using~\cite[Chapter VI, $\S 8$, Theorem 8.3]{Neukirch}, we get
	\[
	\left(\frac{a}{b}\right)_{M_q,4}=\left(\frac{b}{a}\right)_{M_q,4}\cdot \left(\frac{a,b}{\mathfrak{I}}\right)_{M_q,4}.
	\]
where $\mathfrak{I}$ denotes the ideal $(1+i)$ of $M_q$. Note that the infinite places do not contribute in this product, since the field $M_q$ is totally complex.

We prove that $\left(\frac{a,b}{\mathfrak{I}}\right)_{M_q,4}$ depends only on $a$ and $b$ modulo 32.
If $a \equiv 1\bmod 32,$ where $a\in \mathcal{O}_{(M_q)_{\mathfrak{I}}},$ then $a$ is a fourth power in $(M_q)_{\mathfrak{I}}$ by Hensel's lemma. So we deduce that $\left(\frac{a,b}{\mathfrak{I}}\right)_{M_q,4}=1$ applying property $(iii)$ of Proposition~\ref{properties of Hilbert}. If $b$ is congruent to 1 modulo 32, then we get the same result using properties $(iii)$ and $(iv)$ of Proposition~\ref{properties of Hilbert}.
If neither $a$ nor $b$ is congruent to 1 modulo 32, let $a'$ and $b'$ be different from $a$ and $b$ respectively and such that $a\equiv a'\bmod 32$ and $b \equiv b' \bmod 32.$ Then $a=\gamma a'$ and $b =\tilde{\gamma}b'$ with $\gamma, \tilde{\gamma}$ congruent to 1 modulo 32. Using properties $(i)$ and $(ii)$ of Proposition~\ref{properties of Hilbert}, we get 
\begin{align*}
	\left(\frac{a,b}{\mathfrak{I}}\right)_{M_q,4}&=\left(\frac{\gamma a', \tilde{\gamma}b'}{\mathfrak{I}}\right)_{M_q,4}\\&=\left(\frac{\gamma,\tilde{\gamma}b'}{\mathfrak{I}}\right)_{M_q,4}\left(\frac{a',\tilde{\gamma}}{\mathfrak{I}}\right)_{M_q,4}\left(\frac{a',b'}{\mathfrak{I}}\right)_{M_q,4}\\&=\left(\frac{a',b'}{\mathfrak{I}}\right)_{M_q,4}.
\end{align*}
In the case of $q=7$, we have two different prime ideals $\mathfrak{I}_1$ and $\mathfrak{I}_2$, in $M_7$ above 2. So we have
	\[
	\left(\frac{a}{b}\right)_{M_7,4}=\left(\frac{b}{a}\right)_{M_7,4} \left(\frac{a,b}{\mathfrak{I}_1}\right)_{M_7,4}\left(\frac{a,b}{\mathfrak{I}_2}\right)_{M_7,4}.
	\]
Nonetheless, we can use the very same argument we used before, taking into account that we have two different prime ideals above 2 instead of just one.

Now, let us prove that $(a/b)_{M_q,4}$ depends only on the congruence class of $b$ modulo $32a\mathcal{O}_{M_q}.$ Using~\cite[Chapter VI, $\S 8$, Theorem 8.3]{Neukirch}, we obtain
	\[
	\left(\frac{a}{b}\right)_{M_q,4}=\prod_{\mathfrak{p}\not \in S(a)} \left(\frac{b, a}{\mathfrak{p}}\right)_{M_q,4}=\prod_{\mathfrak{p}\in S(a)} \left(\frac{ a, b}{\mathfrak{p}}\right)_{M_q,4},
	\]
where $S(a):=\{\mathfrak{p}:\mathfrak{p}\mid n\cdot \infty \text{ or }\ord_\mathfrak{p}(a)\neq 0\}.$

As for the prime ideal $\mathfrak{I}$, we already saw that $\left(\frac{a, b}{\mathfrak{I}}\right)_{M_q, 4}$ depends only on $b$ modulo 32 (and the same holds for $\mathfrak{I}_1$ and $\mathfrak{I}_2$ in the case of $q=7$). If $\mathfrak{p}\in S(a)$ is odd, we have 
	\[
	\left(\frac{ a, b}{\mathfrak{p}}\right)_{M_q,4}=\left(\frac{b}{\mathfrak{p}}\right)_{M_q,4}^{\ord_\mathfrak{p}(a)}.
	\]
Hence the value of these symbols depends only on $b$ modulo $a$. Therefore the total symbol depends only on $b$ modulo $32a$.
\end{proof}


\subsection{Field lowering}

For the reader's convenience, we state three lemmas that we will use in the proof of Theorem~\ref{Result}, reducing the quartic residue symbol in a quartic number field to a quadratic residue symbol in a quadratic number field.
These lemmas are stated and proved in~\cite[\textsection 3.2]{16-rank}.
\begin{lem}
\label{p splits}
Let $K$ be a number field and let $\mathfrak{p}$ be an odd prime ideal of $\mathcal{O}_K$. Suppose that $L$ is a quadratic extension of $K$ such that $L$ contains $\mathbb{Q}(\sqrt{-1})$ and $\mathfrak{p}$ splits in $L$. 
Denote by $\psi$ the non-trivial element in $\Gal(L/K)$. Then if $\psi$ fixes $\mathbb{Q}(\sqrt{-1})$, we have for all $\alpha\in \mathcal{O}_K$
	\[
	\left(\frac{\alpha}{\mathfrak{p}\mathcal{O}_L}\right)_{L,4}=\left(\frac{\alpha}{\mathfrak{p}\mathcal{O}_K}\right)_{K,2}
	\]
and if $\psi$ does not fix $\mathbb{Q}(\sqrt{-1})$ we have for all $\alpha\in \mathcal{O}_K$ with $\mathfrak{p}\nmid \alpha$
	\[
	\left(\frac{\alpha}{\mathfrak{p}\mathcal{O}_L}\right)_{L,4}=1.
	\]
\end{lem}
\begin{lem}
\label{p inerts}
Let $K$ be a number field and let $\mathfrak{p}$ be an odd prime ideal of $\mathcal{O}_K$ of degree 1 lying above $p.$ Suppose that $L$ is a quadratic extension of $K$ such that $L$ contains $i$ and $\mathfrak{p}$ stays inert in $L$. We have for all $\alpha\in \mathcal{O}_K$
	\[
	\left(\frac{\alpha}{\mathfrak{p}\mathcal{O}_L}\right)_{L,4}=\left(\frac{\alpha}{\mathfrak{p}\mathcal{O}_K}\right)_{K,2}^{\frac{p+1}{2}}.
	\]
\end{lem}
\begin{lem}
\label{p not ramifies}
Let $K$ be a number field and let $L$ be a quadratic extension of $K$. Denote by $\psi$ the non-trivial element in $\Gal(L/K)$. Suppose that $\mathfrak{p}$ is a prime ideal of $\mathcal{O}_K$ that does not ramify in $L$ and further suppose that $\beta\in \mathcal{O}_L$ satisfies $\beta \equiv \psi(\beta)\bmod \mathfrak{p}\mathcal{O}_L$. Then there is $\beta'\in \mathcal{O}_K$ such that $\beta'\equiv \beta \bmod \mathfrak{p}\mathcal{O}_L.$
\end{lem}


\section{The 2-part of the class group}

Let $k\geqslant1$ be an integer. The $2^k$-rank of a finite abelian group $G$, denoted by rk$_{2^k}G$, is the dimension of the $\mathbb{F}_2$-vector space $2^{k-1}G/2^kG.$
If the 2-Sylow subgroup of $G$ is cyclic, we have $\rk_{2^k}G \in \{0,1\}$ and $\rk_{2^k}G=1$ if and only if $2^k|\#G$.
We will study the necessary and sufficient conditions such that $2^k\mid h(-qp)$ for $k\in\{1,2,3,4\}.$
Moreover, for each integer $k\geqslant 1$ and fixed $q\in Q$ with $Q=\{3,7,11,19,43,67,163\}$, we define a density $\delta(2^k)$ as 
	\[
	\delta(2^k):=\lim_{x\to \infty} \frac{\#\{p\leqslant x:\,p\equiv 1\bmod 4,\,\, 2^k\mid\#\Cl(K_{p,q})\}}{\#\{p\leqslant x:\,p\equiv 1\bmod 4\}},
	\]
if the limit exists.


\subsection{The 2-rank}

The discriminant $D_{K_{p,q}}$ of the extension $K_{p,q}=\mathbb{Q}(\sqrt{-qp})$ is equal to $-qp$, where $q\in Q$ and $p$ a prime congruent to 1 modulo 4.
Then, by Gauss genus theory, we have that $|\Cl(K_{p,q})[2]|=2$ and so $\delta(2)=1$. In particular, it follows that the 2-Sylow subgroup $\Cl(K_{p,q})[2^\infty]$ of the class group is cyclic, as it is an abelian 2-group with just one non-trivial element of order 2.
We describe it as 
	\[
	\Cl(K_{p,q})[2]=\langle[ \mathfrak{t}],[\mathfrak{p}]\rangle,
	\]
where $\mathfrak{t}$ is the prime ideal above $q$ and $\mathfrak{p}$ is the prime ideal above $p$ in $K_{p,q}$.


\subsection{The 4-rank}

For the 4-rank of the class group of $K_{p,q}$, we look for an element of order 4.
We have that $\rk_4\Cl(K_{p,q})=1$ if and only if the map
	\[
	\varphi:\Cl(K_{p,q})[2]\longrightarrow \Cl(K_{p,q})/2\Cl(K_{p,q})
	\]
is the zero map.
By class field theory, the genus field $H_2$ is the field $K_{p,q}(\sqrt{-q})$ and we have
	\[
	\Cl(K_{p,q})/2\Cl(K_{p,q})\cong \Gal(H_2/K_{p,q}).
	\]
So the map $\varphi$ is trivial if and only if the Artin symbol corresponding to $\mathfrak{p}$, the prime ideal above $p$, (or analogously the one corresponding to $\mathfrak{t}$) is trivial. It is equivalent to say that $\mathfrak{p}$ (or analogously $\mathfrak{t}$) splits completely in $K_{p,q}\subset H_2$. This is the same as asking that $p$ splits completely in $\mathbb{Q}(\sqrt{-q})$ (or analogously that $q$ splits completely in $\mathbb{Q}(\sqrt{p})$).
Then we have 
	\[
	4\mid h(-qp) \Leftrightarrow \left(\frac{-q}{p}\right)=1\Leftrightarrow \left(\frac{p}{q}\right)=1.
	\]
So, the 4-rank is 1 if and only if $p$ splits completely in $\mathbb{Q}(\sqrt{-1}, \sqrt{-q})$.
Using the Chebotarev Density Theorem, we obtain $\delta(4)=\frac{1}{2}$.


\subsection{The 8-rank}

We have an element of order 8 in the class group if and only if the map
	\[
	\psi:\Cl(K_{p,q})[2]\longrightarrow \Cl(K_{p,q})/4\Cl(K_{p,q})
	\]
is the zero map. Again, by class field theory, we have an extension $H_4$ of $K_{p,q}$, called the 4-Hilbert class field, that is contained in the Hilbert class field $H(K_{p,q})$. The field $H_4$ is such that $\Gal(H_4/K_{p,q})\cong \Cl(K_{p,q})/4\Cl(K_{p,q}).$ The map $\psi$ is trivial if and only if the Artin symbol of $\mathfrak{p}$ (resp. of $\mathfrak{t}$) of the extension $K_{p,q}\subset H_4$ is trivial that corresponds to ask that $\mathfrak{p}$ (resp. $\mathfrak{t}$) splits completely in $H_4$. We choose to work with the prime $q$, but it is symmetric to the prime $p$.

Since $q$ ramifies in $\mathbb{Q}(\sqrt{-q})$, it is equivalent to ask that $(\sqrt{-q})$ splits completely in the extension $\mathbb{Q}(\sqrt{-q})\subset H_4$. Let us call $F:=\mathbb{Q}(\sqrt{-q}).$ We have
\begin{displaymath}
\xymatrix{& H_4 \ar@{-}[d] \\ 
& H_2= \mathbb{Q}(\sqrt{-q},\sqrt{p})
\ar@{-}[rd] \ar@{-}[d] \ar@{-}[ld]\\
F=\mathbb{Q}(\sqrt{-q}) \ar@{-}[rd] & 
K_{p,q}=\mathbb{Q}(\sqrt{-qp})\ar@{-}[d] & \mathbb{Q}(\sqrt{p})\ar@{-}[ld] \\
& \mathbb{Q}}
\end{displaymath}
The extension $F\subset H_4$ is abelian of order 4 and exponent 2. The only primes that ramify are the ones over $p$, say $\mathfrak{p}_1$ and $\mathfrak{p}_2$, they are tamely ramified of ramification index 2. The conductor of this extension is $p$. Let $\Cl_p(F)$ be the ray class group with respect to the conductor $p$. Recall that we have an exact sequence of finite abelian groups 
	\[
	0 \longrightarrow \left(\mathcal{O}_F/p\,\mathcal{O}_F\right)^\times/\text{Im} (\mathcal{O}_F^\times)\longrightarrow \Cl_p(F)\longrightarrow \Cl(F)\longrightarrow 0
	\]
and, since $\Cl(F)=1$, we have the following isomorphism
	\[
	\left(\mathcal{O}_F/p\,\mathcal{O}_F\right)^\times/\text{Im} (\mathcal{O}_F^\times)\cong\Cl_p(F).
	\]
Note that $\left(\mathcal{O}_F/p\,\mathcal{O}_F\right)^\times\cong \left(\mathcal{O}_F/\mathfrak{p}_1\,\mathcal{O}_F\right)^\times\times\left(\mathcal{O}_F/\mathfrak{p}_2\,\mathcal{O}_F\right)^\times $ and for $i\in \{1,2\}$, each factor
$\left(\mathcal{O}_F/\mathfrak{p}_i\,\mathcal{O}_F\right)^\times$ is isomorphic to $\mathbb{F}_p^\times$.
The Artin map ensures us the existence of a surjection
	\[
	\Cl_p(F)\twoheadrightarrow\Gal(H_4/F)\cong C_2\times C_2,
	\]
that sends a prime ideal of $\Cl_p(F)$ onto its Artin symbol. 
Then, if we quotient by $2\Cl_p(F)$, we get the following isomorphisms
\begin{equation}
	\label{clp}
	(\mathbb{F}_p^\times\times\mathbb{F}_p^\times)/\square\cong \Cl_p(F)/2\Cl_p(F)\cong \Gal(H_4/F)\cong C_2\times C_2.
\end{equation}

In order to have that $(\sqrt{-q})$ splits completely in $H_4$, we want that its Artin symbol is trivial. Hence, considering~\eqref{clp}, we need that $\sqrt{-q}$ is a square modulo $p$. Therefore, if $p$ is a prime congruent to 1 modulo 4 and such that $(-q/p)=1$, we have the following condition
\begin{equation}
	\label{div 8}
	8\mid h(-qp) \Leftrightarrow \left(\frac{-q}{p}\right)_4=1,
\end{equation}
where the quartic symbol is for $K=\mathbb{Q}$.

Note that~\eqref{div 8} is equivalent to $p$ splitting completely in $\mathbb{Q}(\sqrt{-1},\sqrt[4]{-q})$. 
Indeed if we consider the following extensions
\begin{displaymath}
\xymatrix{
& \mathbb{Q}(\sqrt{-1},\sqrt[4]{-q}) \ar@{-}[d] \\ & \mathbb{Q}(\sqrt{-1},\sqrt{-q}) \ar@{-}[d]\\
& \mathbb{Q}}
\end{displaymath}
we have that $p$ splits completely in $\mathbb{Q}(\sqrt{-1},\sqrt{-q})$, since $(-q/p)=1.$
If $\mathfrak{p}$ is a prime ideal in $\mathbb{Q}(\sqrt{-1},\sqrt{-q})$ above $p,$ then we have that $\mathfrak{p}$ splits completely in $\mathbb{Q}(\sqrt{-1},\sqrt[4]{-q})$ if and only if $p$ splits completely in $\mathbb{Q}(\sqrt{-1},\sqrt[4]{-q})$ if and only if $(-q/p)_4=1.$

We know that $\mathbb{Q}(\sqrt{-1},\sqrt{-q})$ is a principal ideal domain and so if $\pi$ is a generator of a prime ideal $\mathfrak{p}$ in $\mathbb{Q}(\sqrt{-1},\sqrt{-q})$ above $p$, since $\mathfrak{p}$ has degree 1, we have that 
	\[
	\left(\frac{-q}{p}\right)_{\mathbb{Q},\,4}=\left(\frac{-q}{\pi}\right)_{\mathbb{Q}(\sqrt{-1},\sqrt{-q}),\,4}.
	\]
Using again the Chebotarev Density Theorem, we obtain $\delta(8)=\frac{1}{4}$.


\subsection{The 16-rank}

The criterion for the divisibility of $h(-qp)$ by 16 is due to Leonard and Williams~\cite[Theorem 2]{Leonard and Williams}.
Let $p$ be a prime number congruent to 1 modulo 4, such that $\left(\frac{-q}{p}\right)=1$ and $\left(\frac{-q}{p}\right)_4=1$. There exist positive integers $u$ and $v$ satisfying $p=u^2-qv^2$. We will show that we can always find a solution with $u\equiv 1 \bmod 4$. Then
	\[
	16\mid h(-qp)\quad \Leftrightarrow\quad\left(\frac{u}{p}\right)_4=\left(\frac{2}{u}\right), \quad 
	\begin{array}{ll} \text{where } u\equiv 1 \bmod 4 \\ \text{and }p=u^2-qv^2, \,\,u,v\in \mathbb{Z}.
	\end{array}
	\]
We note that the first quartic symbol has both entries depending on $p$, since $u$ has to satisfy the relation $p=u^2-qv^2.$
Hence, we cannot interpret this condition as the splitting behaviour of $p$ in some normal extension of $\mathbb{Q}$ and thus we cannot directly apply the Chebotarev Density Theorem, as we did before. Instead, we will follow Koymans and Milovic's idea using Vinogradov's method.

Note that $u$ and $v$ are not uniquely determined. Let us see how we can compute these integers.
It is natural to work in the field $\mathbb{Q}(\sqrt{q})$. We observe that $p$ splits completely in $M_q=\mathbb{Q}(\sqrt{-1}, \sqrt{q})$, since we have $(-q/p)=1.$
We already know that $M_q$ is a principal ideal domain. Let $\zeta_{12}$ be a 12-th root of unity and $i=\sqrt{-1}$ be a fourth root of unity. We see that $ \mathcal{O}^\times_{M_q}=\langle \nu_q\rangle\times\langle\varepsilon_q\rangle,$ where
	\[
	\nu_q=\begin{cases}\zeta_{12}&\text{if }q=3,\\
	i&\text{otherwise,}
	\end{cases}
	\]
and
\begin{equation}
	\label{epsilon}
	\begin{aligned}
	\varepsilon_3&=\zeta_{12}-1,\\\varepsilon_7&=\frac{1}{2}(1-i)(\sqrt{7}+3),\\\varepsilon_{11}&=\frac{1}{2}(1-i)(\sqrt{11}+3),\\\varepsilon_{19}&=\frac{1}{2}(1+i)(3\sqrt{19}+13),\\\varepsilon_{43}&=\frac{1}{2}(1+i)(9\sqrt{43}-59),\\\varepsilon_{67}&=\frac{1}{2}(1+i)(27\sqrt{67}-221),\\\varepsilon_{163}&=\frac{1}{2}(1-i)(627\sqrt{163}+8005).
	\end{aligned}
\end{equation}
Note that $M_{q}/\mathbb{Q}$ is a normal extension with Galois group isomorphic to the Klein four group, say $\{1, \sigma ,\tau, \sigma\tau\}$, where $\sigma$ fixes $\mathbb{Q}(\sqrt{q})$ and $\tau$ fixes $\mathbb{Q}(\sqrt{-1}).$
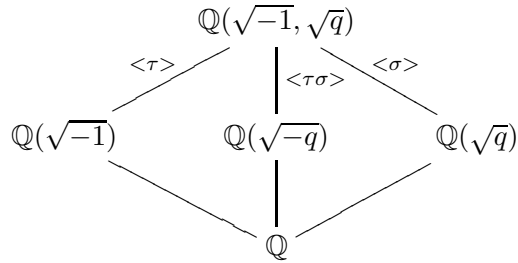
\begin{displaymath}
\xymatrix{ & \mathbb{Q}(\sqrt{-1}, \sqrt{q})\ar@{-}[rd]^{<\sigma>} \ar@{-}[d]^{<\tau\sigma>} \ar@{-}[ld]_{<\tau>}\\
 \mathbb{Q}(\sqrt{-1}) \ar@{-}[rd] &  \mathbb{Q}(\sqrt{-q})\ar@{-}[d] & \mathbb{Q}(\sqrt{q})\ar@{-}[ld] \\
& \mathbb{Q}}
\end{displaymath}
We consider $\pi \in M_q$ such that $\pi $ generates one of the prime ideals $\mathfrak{p}$ in $\mathcal{O}_{M_q}$ above $p.$ Then there exists $u$, $v\in \mathbb{Z}$ such that $u+\sqrt{q} v=\mathrm{N}_{M_q/\mathbb{Q}(\sqrt{q})}(\pi)$ and so we get
	\[
	\pm p=\mathrm{N}_{M_q/\mathbb{Q}}(\pi)=(u+\sqrt{q} v)(u-\sqrt{q}v)=u^2-qv^2.
	\]
Looking at this equation mod 4, we have
\begin{equation}
\label{u and v equation}
p=u^2-qv^2,
\end{equation}
as wanted.
Thus we can choose
	\[
	u=\frac{\pi\sigma(\pi)+\tau(\pi)\tau(\sigma(\pi))}{2} \quad \text{and} \quad v=\frac{\pi\sigma(\pi)-\tau(\pi)\tau(\sigma(\pi))}{2\sqrt{q}}.
	\]
We now check that $u>0$ and that we can always find $u\equiv 1 \bmod 4$. In fact, if $u_0$ and $v_0$ are a solution of~\eqref{u and v equation}, then also $u+\sqrt{q}v=(u_0+\sqrt{q}v_0)(\sigma(\varepsilon_q)\varepsilon_q)^k$, for $k\in \mathbb{N}$, is a solution. Indeed 
\begin{align*}
	\mathrm{N}_{M/\mathbb{Q}(\sqrt{q})}(\pi)= \mathrm{N}_{M/\mathbb{Q}(\sqrt{q})}(\varepsilon_q\pi)&=\sigma(\varepsilon_q)\varepsilon_q(u_0+\sqrt{q}v_0).
\end{align*}
The map that describes the transformation of a given solution $(u,v)$ for the equation~\eqref{u and v equation}, by the multiplication with $\sigma(\varepsilon_q)\varepsilon_q$ modulo 4, is the following
\begin{equation}
\label{orbite}
	\mathbb{Z}/4\mathbb{Z}\times \mathbb{Z}/4\mathbb{Z}\longrightarrow\mathbb{Z}/4\mathbb{Z}\times \mathbb{Z}/4\mathbb{Z},
\end{equation}
\begin{align*}
	\quad\quad\quad\quad(u,v)&\longmapsto(2u-3v,2v-u),\quad \text{if }q=3,11, 19,163,\\\quad\quad\quad\quad(u,v)&\longmapsto(v,3u),\quad\quad\quad\quad \quad\,\,\,\,\text{if }q=7,\\
	\quad\quad\quad\quad(u,v)&\longmapsto(2u+3v,u+2v),\quad\, \text{if }q=43,67.
\end{align*}
The possibilities for $u$ and $v$ are $u=0,2$ and $v=1,3$, or $u=1,3$ and $v=0,2$ modulo 4 and so, looking at the orbits of the afore defined map, we note that those are of length 4 and that we always find exactly  one $u$ in each orbit that satisfies $u\equiv 1 \bmod 4$.


\subsection{Encoding the 16-rank of \texorpdfstring{$\Cl(K_{p,q})$}{Cl(K_{p,q})} into sequences \texorpdfstring{$\{a_{\mathfrak{n},q}\}_{\mathfrak{n}}$}{a_{n,q}}}

Let $q$ be a fixed prime in the set $Q$ and let $n_q$ be equal to $2q$. We define, for any $\alpha \in \mathbb{Z}[\sqrt{q}]$,
	\[
	\texttt{u}(\alpha)=\frac{1}{2}(\alpha+\tau(\alpha)).
	\]
Note that for every $w\in \mathcal{O}_{M_q}\setminus \{0\}$ the inequality $\texttt{u}(w\sigma(w))>0$ holds.
We define for any element $w\in \mathcal{O}_{M_q}$ coprime with $n_q$, 
\begin{equation}
	\label{[]}
	[w]:=\left(\frac{\texttt{u}(w\sigma(w))}{w}\right)_{M_q,4}\left(\frac{2}{\texttt{u}(w\sigma(w))}\right)_{\mathbb{Q},2}.
\end{equation} 
Hence $16\mid h(-qp)$ if and only if $[w]=1$, where $w$ is any element of $\mathcal{O}_{M_q}$ such that $N_{{M_q}/\mathbb{Q}}(w)=p$ and $\texttt{u}(w\sigma(w))\equiv 1 \bmod 4$.
We note that
\begin{equation}
	\label{8}
	\begin{aligned}
	\left(\frac{\texttt{u}(w\sigma(w))}{w}\right)_{{M_q},4}&=\left(\frac{(w\sigma(w)+ \tau(w\sigma(w)))/2}{w}\right)_{{M_q},4}\\&=\left(\frac{\tau(w\sigma(w))}{w}\right)_{{M_q},4}\left(\frac{8}{w}\right)_{{M_q},4}.
	\end{aligned}
\end{equation}
Let $\varepsilon_q$ be as in~\eqref{epsilon}. Then
\begin{align*}
	\left(\frac{\text{u}(\varepsilon^4 w\sigma(\varepsilon_q^4w))}{\varepsilon_q^4 w}\right)_{{M_q},4}&=\left(\frac{\text{u}(\varepsilon_q^4 w\sigma(\varepsilon_q^4 w))}{w}\right)_{{M_q},4}\\&=\left(\frac{\tau(\varepsilon_q^4 w\sigma(\varepsilon_q^4w))}{w}\right)_{{M_q},4}\left(\frac{8}{w}\right)_{{M_q},4}\\
	&=\left(\frac{\text{u}(w\sigma(w))}{w}\right)_{{M_q},4}.
\end{align*}
The equality
	\[
	\left(\frac{2}{\texttt{u}(\varepsilon_q^4w\sigma(\varepsilon_q^4w))}\right)_{\mathbb{Q},2}=\left(\frac{2}{\texttt{u}(w\sigma(w))}\right)_{\mathbb{Q},2},
	\]
is given by $2\texttt{u}(\varepsilon_q^4w\sigma(\varepsilon_q^4w)) \equiv 2\texttt{u}(w\sigma(w))\bmod 16.$
In fact, we have that $2\texttt{u}(\varepsilon_q^4w\sigma(\varepsilon_q^4w))=(\varepsilon_q\sigma(\varepsilon_q))^4w\sigma(w)+\tau(\varepsilon_q\sigma(\varepsilon_q))^4\tau(w\sigma(w))$ and a straightforward computation (with $w\sigma(w)=u+\sqrt{q}v$) shows that
\begin{align*}
	2\texttt{u}(\varepsilon_3^4 w\sigma(\varepsilon_3^4w)) &= 194\,u - 336 \,v,\\
	2\texttt{u}(\varepsilon_7^4w \sigma(\varepsilon_7^4w)) &= 64514\, u + 170688\,v,\\
	2\texttt{u}(\varepsilon_{11}^4w \sigma(\varepsilon_{11}^4w)) &= 158402 \,u + 525360\,v,\\
	2\texttt{u}(\varepsilon_{19}^4w \sigma(\varepsilon_{19}^4w)) &= 13362897602\, u + 58247520240\,v,\\
	2\texttt{u}(\varepsilon_{43}^4w \sigma(\varepsilon_{43}^4w)) &= 2351987525322434\, u - 15423013607227056\,v,\\
	2\texttt{u}(\varepsilon_{67}^4w \sigma(\varepsilon_{67}^4w)) &= 91052891016584133314\, u -745300033869597034608\,v,\\
	2\texttt{u}(\varepsilon_{163}^4w \sigma(\varepsilon_{163}^4w)) &= 269780589805913908506459977860802\, u \\
	&\quad+ 3444327998561165640260096561357040\,v.
\end{align*}
Hence we proved that
	\[
	[w]=[\varepsilon_q^4 w].
	\]
Note that 
	\[
	[w]=[\nu_q w].
	\]
Indeed for $q=3$, we see that $\zeta_{12}\sigma(\zeta_{12})=1$ and hence $\tau(\zeta_{12}\sigma(\zeta_{12}))=1$. Then $\texttt{u}(\zeta_{12}w\,\sigma(\zeta_{12}w))=\texttt{u}(w\sigma(w))$. For the other values of $q$, note that $i\sigma(i)=1$ and so $\tau(i\sigma(i))=1$. Then $\texttt{u}(iw\,\sigma(iw))=\texttt{u}(w\sigma(w)).$\newline
For $w\in \mathcal{O}_{M_q}$ such that $N_{{M_q}/\mathbb{Q}}(w)=p$, we define
	\[
	s(w)=
	\begin{cases} 1& \text{if } \texttt{u}(w\sigma(w))\equiv 1 \bmod 4,\\
	0&\text{otherwise.}
	\end{cases}
	\]
Then
	\[
	\sum_{i=0}^3 s(\varepsilon_q^i w)=1,
	\]
with $\varepsilon_q$ as in~\eqref{epsilon}, by looking at the orbits of the map~\eqref{orbite}.
Thus it is clear that 
	\[
	\sum_{i=0}^3 s(\varepsilon_q^i w)=\sum_{i=0}^3 s(\varepsilon_q^{i+k} w),
	\]
where $k\in \mathbb{N}.$

Having determined the action of the units $\mathcal{O}_{M_q}^\times$, on this sum and on $[\,\cdot\,], $ we see that the quantity $\sum_{i=0}^3s(\varepsilon_q^iw)[\varepsilon_q^i w]$ does not depend on the choice of the generator $w$ but only on the prime ideal $\mathfrak{p}$ above $p$. We have proved the following.
\begin{prop}
Let $p$ be a prime congruent to 1 modulo 4 and such that $(-q/p)=1$. Let $\mathfrak{p}$ be a prime ideal of $\mathcal{O}_{M_q}$, lying above p, with generator $w$. Then
	\[
	\sum_{i=0}^3s(\varepsilon_q^iw)[\varepsilon_q^i w]\frac{1}{2}\left(1+\left(\frac{-q}{\varepsilon_q^i w}\right)_{M_q,4}\right)=\begin{cases}1&\text{if }16\mid h(-qp),\\
	-1&\text{if } 8\mid h(-qp)\text{ but }16\nmid\,h(-qp),\\0&\text{otherwise.}
	\end{cases}
	\]
\end{prop}

\noindent We define the sequence $\{a_{\mathfrak{n},q}\}_\mathfrak{n}$, indexed by ideals of $\mathcal{O}_{M_q}$, in the following way
\begin{equation}
	\label{sequence}
	a_{\mathfrak{n},q}:=\begin{cases}0&\text{if }(\mathfrak{n},n_q)=1,\\
	\sum_{i=0}^3s(\varepsilon_q^iw)[\varepsilon_q^i w]\frac{1}{2}\left(1+\left(\frac{-q}{\varepsilon_q^i w}\right)_{M_q,4}\right)&\text{otherwise,}
	\end{cases}
\end{equation}
where $w$ is any generator of the ideal $\mathfrak{n}$ coprime to $n_q$.


\section{Vinogradov’s method, after Friedlander, Iwaniec, Mazur and Rubin}

The version of Vinogradov's method that we are going to use is the one introduced by Friedlander \emph{et al\@.}\ in~\cite{spin of prime ideals}.
In order to use this powerful machinery, we need to verify that the sequence $(a_{\mathfrak{n},q})$ defined in~\eqref{sequence}, satisfies the hypothesis of~\cite[Proposition 5.2]{spin of prime ideals}. In other words, it remains to prove analogues of Propositions 3.7 and 3.8 of~\cite{16-rank} for our sequences $(a_{\mathfrak{n},q})$ with $q\in Q$ fixed and the field $M_q$, where $Q=\{3,7,11,19, 43,67,163\}$. In the literature, the sums that will appear are called sums of type I and sums of type II respectively.\newline
Once we have proved it, we will have that
	\[
	\sum_{\mathrm{N}\mathfrak{n}\leqslant x}a_{\mathfrak{n},q}\Lambda (\mathfrak{n})\ll_{\theta}x^{1-\theta}, \quad \quad \text{with } x>0,
	\]
for all $\theta<1/(49\cdot 64)=1/3136.$
This implies Theorem~\ref{Result}.


\subsection{Sums of type I}

In this section, we will adapt the proof of~\cite[Proposition 3.7]{16-rank} for our sequence $(a_{\mathfrak{n},q})$ and the field $M_q$ for $q$ fixed.

\noindent Let $\mathfrak{m}$ be an ideal of $\mathcal{O}_{M_q}$ coprime with $n_q$. We want to bound the following sum
	\[
	\begin{split}
	A(x)=A(x,\mathfrak{m}):=&\sum_{\substack{\mathrm{N}(\mathfrak{n})\leqslant x\\(\mathfrak{n},n_q)=1, \,\mathfrak{m}\mid\mathfrak{n}}}a_{\mathfrak{n},q}\\
	=&\sum_{\substack{\mathrm{N}(\mathfrak{n})\leqslant x\\
	(\mathfrak{n},n_q)=1, \,\mathfrak{m}\mid\mathfrak{n}}}\left(\sum_{i=0}^3s(\varepsilon_q^i\alpha)[\varepsilon_q^i \alpha]\frac{1}{2}\left(1+\left(\frac{-q}{\varepsilon_q^i \alpha}\right)_{M_q,4}\right)\right),
	\end{split}
	\]
where $\alpha$ is a generator of $\mathfrak{n}$.
We consider the integral basis $\{1,\eta_{1}^{(q)} ,\eta_{2}^{(q)},\eta_{3}^{(q)}\}$ (e.g. for $q=3$ we consider $\{1,\zeta_{12} ,\zeta_{12}^2,\zeta_{12}^3\}$) and a fundamental domain $\mathcal{D}_q$ as described in~\cite[Lemma 3.5]{16-rank} with $F=M_q$ and $n=4.$

In the case $q=3$, the torsion group of $\mathcal{O}_{M_3}^\times$, has twelve elements and then every ideal $ \mathfrak{n}$ has exactly twelve generators $\alpha\in \mathcal{D}_3.$ For the other cases, the torsion part of the unit group $\mathcal{O}_{M_q}^\times$ has four elements and so every ideal $ \mathfrak{n}$ has exactly four generators $\alpha\in \mathcal{D}_q.$ We recall that $s(\alpha)$ depends only on the congruence class of $\alpha$ modulo 4. See that 
\begin{align*}
	[\alpha]=&\left(\frac{\texttt{u}(\alpha\sigma(\alpha))}{\alpha}\right)_{M_q,4}\left(\frac{2}{\texttt{u}(\alpha\sigma(\alpha))}\right)_{\mathbb{Q},2}\\
	=&\left(\frac{\tau(\alpha)}{\alpha}\right)_{M_q,4}\left(\frac{\tau(\sigma(\alpha))}{\alpha}\right)_{M_q,4}\left(\frac{8}{\alpha}\right)_{M_q,4}\left(\frac{2}{\texttt{u}(\alpha\sigma(\alpha))}\right)_{\mathbb{Q},2}.
\end{align*}
The symbol $\left(\frac{2}{\texttt{u}(\alpha\sigma(\alpha))}\right)_{\mathbb{Q},2}$ depends only on the congruence class of $\alpha$ modulo 8 and, by Lemma~\ref{reciprocità quartica}, the symbol $\left(8/\alpha\right)_{M_q,4}$ depends only on the congruence class of $\alpha$ modulo $2^8.$ 
We set $F_q=q\cdot 2^8$ and we split the sum $A(x)$ into congruence classes modulo $F_q$. 

\noindent Using Lemma~\ref{reciprocità quartica}, the symbol $\left(-q/\alpha\right)_{M_q,4}$ depends only on $\alpha$ modulo $32q$ and  so we find that 
	\[
	A(x)=\frac{1}{12}\sum_{i=0}^3\sum_{\substack{\rho \bmod F_3\\
	(\rho,F_3)=1}}\,\,\frac{1}{2}\,\,\mu(\rho\varepsilon_3^i) A(x;\rho,\varepsilon_3^i) \left(1+\left(\frac{-3}{\varepsilon_3^i \alpha}\right)_{M_3,4}\right)\,\,\, \text{for }q=3,\]
and
	\[
	A(x)=\frac{1}{4}\sum_{i=0}^3\sum_{\substack{\rho \bmod F_q\\
	(\rho,F_q)=1}}\,\,\frac{1}{2}\,\,\mu(\rho,\varepsilon_q^i)A(x;\rho,\varepsilon_q^i)\left(1+\left(\frac{-q}{\varepsilon_q^i \alpha}\right)_{M_q,4}\right)\,\,\, \text{  otherwise,}
	\] 
where $\mu(\rho,\varepsilon_q^i)\in \{\pm 1,\pm i\}$ depends only on $\rho$ and $\varepsilon_q^i$ and with
	\[
	A(x;\rho,\varepsilon_q^i):=\sum_{\substack{\alpha\in \varepsilon_q^i\mathcal{D}_q,\, \mathrm{N}(\alpha)\leqslant x\\	\alpha\equiv \rho \bmod F_q\\ \alpha \equiv 0 \bmod \mathfrak{m}}}\left(\frac{\tau(\alpha)}{\alpha}\right)_{M_q,4}\left(\frac{\tau(\sigma(\alpha))}{\alpha}\right)_{M_q,4}.
	\]
Since $\left(-q/\alpha\right)_{M_q,4}\in \{0,\pm1,\pm i\}$, we obtain that
	\[
	\left|1+\left(\frac{-q}{\alpha}\right)_{M_q,4}\right|\leqslant2.
	\]
Hence we have the following bound
	\[
	|A(x)|\leqslant\frac{1}{4}\sum_{i=0}^3\sum_{\substack{\rho \bmod F_q\\(\rho,F_q)=1}}\,\,|A(x;\rho,\varepsilon_q^i)|
	\]
for every $q\in Q$.\newline
For each $\varepsilon_q^i$ and congruence class $\rho\bmod F_q$ with $(\rho, F_q)=1,$ we estimate $A(x;\rho, \varepsilon_q^i)$ separately. We consider the free $\mathbb{Z}$-module 
	\[
	\mathbb{M}=\mathbb{Z}\eta_{1}^{(q)}\oplus\mathbb{Z}\eta_{2}^{(q)}\oplus\mathbb{Z}\eta_{3}^{(q)}
	\]
of rank 3 and so we write the decomposition $\mathcal{O}_{M_q}=\mathbb{Z}\oplus \mathbb{M}$ viewing $\mathcal{O}_{M_q}$ as a free $\mathbb{Z}$-module of rank four. We write $\alpha$ uniquely as 
	\[
	\alpha=a+\beta, \quad\text{with } a \in \mathbb{Z},\,\,\beta\in \mathbb{M},
	\]
so our summation conditions become
\begin{equation}
\label{Conditions}
	a+\beta\in \varepsilon_q^i\mathcal{D}_q, \quad N(a+\beta)\leqslant x,\quad a+\beta\equiv \rho \bmod F_q, \quad a+\beta\equiv 0 \bmod \mathfrak{m}.
\end{equation}
If $\tau(\alpha)=\alpha$ and $\tau(\sigma(\alpha))=\alpha,$ we get no contribution to $A(x;\rho,\varepsilon_q^i)$, so we can assume $\tau(\alpha)\neq\alpha$ and $\tau(\sigma(\alpha))\neq\alpha.$
Next we are going to interchange the upper entry and the lower entry of our quartic residue symbols. Since $M_q $ is a principal ideal domain, let
	\[
	\tau(\alpha)-\alpha=\eta^{4} c_0 c\quad \text{and}\quad\tau(\sigma(\alpha))-\alpha=\eta'^{4} c_0' c'
	\]
with $c_0,c_0',c,c',\eta,\eta'\in \mathcal{O}_{M_q}$, $c_0,c_0'\mid F_q$ quadric-free, $\eta, \eta'$ that divide some power of $F_q$ and $(c,F_q)=(c',F_q)=1.$  
We can ensure $c \in \mathbb{Z}[\sqrt{-1}]$ and $c'\in \mathbb{Z}[\frac{1+\sqrt{-q}}{2}]$. In fact, we can take
	\[
	c=\frac{\tau(\alpha)-\alpha}{\sqrt{q}}=\frac{\tau(\beta)-\beta}{\sqrt{q}}\in \mathbb{Z}[\sqrt{-1}]	
	\]
	and
	\[
	c'=\frac{\tau(\sigma(\alpha))-\alpha}{i}=\frac{\tau(\sigma(\beta))-\beta}{i}\in \mathbb{Z}\left[\frac{1+\sqrt{-q}}{2}\right].
	\]
Hence we have 
\begin{align*}
	\left(\frac{\tau(\alpha)}{\alpha}\right)_{M_q,4}&=\left(\frac{a+\tau(\beta)}{\alpha}\right)_{M_q,4}=\left(\frac{\tau(\beta)-\beta}{\alpha}\right)_{M_q,4}\\&=\left(\frac{\eta^{4} c_o c}{\alpha}\right)_{M_q,4}=\left(\frac{c_0}{\alpha}\right)_{M_q,4}\left(\frac{c}{\alpha}\right)_{M_q,4}.
\end{align*}
Since we are working with $\alpha\equiv \rho \bmod F_q$, $(\rho, F_q)=1$ and $c$ and $c'$ depend only on $\beta$, we apply Lemma~\ref{reciprocità quartica} and we obtain
	\[
	\left(\frac{\tau(\alpha)}{\alpha}\right)_{M_q,4}=\tilde{\mu}\cdot\left(\frac{a+\beta}{c\mathcal{O}_{M_q}}\right)_{M_q,4}, 
	\] 
and the same for the other quadric symbol
	\[
	\left(\frac{\tau(\sigma(\alpha))}{\alpha}\right)_{M_q,4}=\tilde{\mu}'\cdot\left(\frac{a+\beta}{c'\mathcal{O}_{M_q}}\right)_{M_q,4},
	\]
with $\tilde{\mu},\tilde{\mu}'\in \{\pm 1, \pm i\}$ that depend only on $\rho$ and $\beta$.
Hence 
	\[
	A(x;\rho, \varepsilon_q^i)\leqslant\sum_{\beta \in \mathbb{M}}|T(x;\beta, \rho, \varepsilon_q^i)|,
	\]
where
	\[
	T(x;\beta, \rho, \varepsilon_q^i):=\sum_{\substack{a\in \mathbb{Z}\\ a+\beta \text{ sat.}\eqref{Conditions}}}\left(\frac{a+\beta}{c\mathcal{O}_{M_q}}\right)_{M_q,4}\left(\frac{a+\beta}{c'\mathcal{O}_{M_q}}\right)_{M_q,4}.
	\]
In order to study $\left(a+\beta/c\mathcal{O}_{M_q}\right)_{M_q,4}$, we want to replace $\beta $ with a rational integer modulo $c \mathcal{O}_{M_q}$.
However this is possible only for ideals of degree 1. For this reason, we factor  $c \mathcal{O}_{M_q}$. Since we choose $c\in \mathbb{Z}[\sqrt{-1}],$ we can define the ideals $\mathfrak{g}$ and $\mathfrak{l}\in \mathbb{Z}[\sqrt{-1}]$ in a unique way such that 
	\[
	(c)=\mathfrak{g}\mathfrak{l}
	\] 
with $l:=\textrm{N}_{\mathbb{Q}(\sqrt{-1})/\mathbb{Q}}(\mathfrak{l})$ a squarefree integer coprime with $n_q$ and $g:=\textrm{N}_{\mathbb{Q}(\sqrt{-1})/\mathbb{Q}}(\mathfrak{g})$ a squarefull integer coprime with $n_ql.$

Note that $c$ is coprime with $2q$. Hence, in the factorization of the ideal $\mathfrak{l}$, all the prime ideals that divide $\mathfrak{l}$ in $\mathbb{Z}[\sqrt{-1}]$ do not ramify in the quadratic extension $M_q.$ We can then apply Lemma~\ref{p not ramifies} for any prime ideal dividing $\mathfrak{l}$ and, using the Chinese Remainder Theorem, we find $\beta' \in \mathbb{Z}[\sqrt{-1}]$ such that $\beta \equiv \beta'\bmod \mathfrak{l}\mathcal{O}_{M_q}$.
We obtain that the upper entry of our quartic residue symbol is in $\mathbb{Z}[\sqrt{-1}].$ 

If a prime ideal $\mathfrak{p}$ that divides $\mathfrak{l}$ and splits in $M_q$, we apply Lemma~\ref{p splits} in order to reduce our quartic symbol to a quadratic one. If $\mathfrak{p}$ stays inert in $M_q$, then we have that $\mathfrak{p}$ has degree 1.
If we define $p:=\mathfrak{p}\cap\mathbb{Z}$, we find that $p\equiv 1 \bmod 4$, since $p$ splits in $\mathbb{Z}[\sqrt{-1}]$, and so $(p+1)/2$ is an odd number. 
Applying Lemma~\ref{p inerts} and combining all these results, we have
	\[
	\left(\frac{\alpha+\beta'}{\mathfrak{l}\mathcal{O}_{M_q}}\right)_{M_q,4}=\left(\frac{\alpha+\beta'}{\mathfrak{l}}\right)_{\mathbb{Q}(\sqrt{-1}),2}.
	\]
Using again the Chinese Remainder Theorem and the fact that $l$ is squarefree, we find a rational integer $b$ such that $\beta'\equiv b\bmod \mathfrak{l}.$ Hence, we have
	\[
	\left(\frac{a+\beta}{c\mathcal{O}_{M_q}}\right)_{M_q,4}=\left(\frac{a+\beta}{\mathfrak{g}\mathcal{O}_{M_q}}\right)_{M_q,4}\left(\frac{a+b}{\mathfrak{l}}\right)_{\mathbb{Q}(\sqrt{-1}),2}.
	\]
Note that $b$ depends on $\beta $ and not on $a$, because $c$ depends only on $\beta.$ We define the product of all the primes dividing $g$ as $g_0:=\prod_{p\mid g}p$ and the product of all the prime ideals dividing $\mathfrak{g}$ as $\mathfrak{g}^*:=\prod_{\mathfrak{p}\mid\mathfrak{g}}\mathfrak{p}.$ The quartic symbol $\left(\alpha/\mathfrak{g}\right)_{M_q,4}$ is periodic in the upper entry modulo $\mathfrak{g}^*,$ and so also modulo $g_0$, since $\mathfrak{g}^*$ divides it.
Since our $\beta$ is fixed, we can split $T(x;\beta, \rho,\varepsilon_q^i)$ into residue classes modulo $g_0$, and we obtain
	\[
	\mid T(x;\beta,\rho,\varepsilon_q^i)\mid\leqslant\sum_{a_0\bmod g_0}\bigg|\sum_{\substack{a\in \mathbb{Z}\\
	a+\beta \text{ sat.}\eqref{Conditions}\\ a\equiv a_0 \bmod g_0}}\left(\frac{a+b}{\mathfrak{q}}\right)_{\mathbb{Q}(\sqrt{-1}),2}\left(\frac{a+\beta}{c'\mathcal{O}_{M_q}}\right)_{M_q,4}\bigg|.
	\]
Now, we focus on the quartic symbol $\left(a+\beta/c'\mathcal{O}_{M_q}\right)_{M_q,4}$. We prove that it is the indicator function for $\gcd(a+\beta,c').$
Note that we have chosen $c'\in\mathbb{Z}\left[\frac{1+\sqrt{-q}}{2}\right] $ and that it is coprime with $n_q$. We factor the principal ideal $(c')\subset\mathbb{Z}\left[\frac{1+\sqrt{-q}}{2}\right]$ as $(c')=\prod_{i=1}^k\mathfrak{p}_i^{e_i}$ where all the $\mathfrak{p}_i$'s are prime ideals of $\mathbb{Z}\left[\frac{1+\sqrt{-q}}{2}\right]$ that do not ramify in $M_q$, since we are sure that they do not divide the discriminant thanks to the coprimality condition with $n_q$. We can then use the definition of quartic residue symbol that we gave and we have
	\[
	\left(\frac{a+\beta}{c'\mathcal{O}_{M_q}}\right)_{M_q,4}=\prod_{i=1}^k\left(\frac{a+\beta}{\mathfrak{p}_i\mathcal{O}_{M_q}}\right)_{M_q,4}^{e_i}.
	\]
To prove that our claim is true, we need to show that $\left((a+\beta)/\mathfrak{p}\mathcal{O}_{M_q}\right)_{M_q,4}=1$ whenever $\mathfrak{p}\nmid a +\beta.$ Using Lemma~\ref{p not ramifies}, 
instead of $\beta$ we can work with $\beta'\in \mathbb{Z}\left[\frac{1+\sqrt{-q}}{2}\right].$ Then we can apply Lemma~\ref{p splits}
for the prime ideals $\mathfrak{p}$ that split in $M_q.$ Instead, if $\mathfrak{p}$ stays inert in $M_q$, we have that $p:=\mathfrak{p}\cap \mathbb{Z}$ has to split in $\mathbb{Q}(\sqrt{-q})$ but not completely in $M_q.$ It follows that $p$ is inert in $\mathbb{Q}(\sqrt{-1})$ and so $(p+1)/2$ is an even number. Then we find that $\mathfrak{p}$ has degree 1 and we conclude our argument with Lemma~\ref{p inerts}.

Hence we obtain
	\[
	\left(\frac{a+\beta}{c'\mathcal{O}_{M_q}}\right)_{M_q,4}=\mathds{1}_{\gcd(a+\beta,c')=(1)}=\sum_{\substack{\mathfrak{d}\mid c'\\ \mathfrak{d}\mid a +\beta}}\mu(\mathfrak{d}),
	\]
where $\mu(\mathfrak{n})$ is the M\"obius function for an integral ideal $\mathfrak{n}$ defined by
	\[
	\mu(\mathfrak{n})=
	\begin{cases}(-1)^t&\text{if }\mathfrak{n}\text{ is the product of $t$ distinct prime ideals,}\\
	\,\,\,\,0&\text{otherwise.}
	\end{cases}
	\]
We obtain 
	\[
	\mid T(x;\beta,\rho,\varepsilon_q^i)\mid\,\,\leqslant\sum_{a_0\bmod g_0}\sum_{\substack{\mathfrak{d}\mid c'\\\mathfrak{d}\text{ squarefree}}}\mid T(x; \beta,\rho,\varepsilon_q^i, a_0,\mathfrak{d}) \mid, 
	\]
with
\begin{equation}
	\label{q somma}
	T(x; \beta,\rho,\varepsilon_q^i, a_0,\mathfrak{d}):=
	\sum_{\substack{a\in \mathbb{Z}\\ a+\beta \text{ sat.}\eqref{Conditions}\\ a\equiv a_0 \bmod g_0\\a+\beta\equiv0 \bmod \mathfrak{d}}}\left(\frac{a+b}{\mathfrak{l}}\right)_{\mathbb{Q}(\sqrt{-1}),2}.
\end{equation}
From now on we can follow the steps of Koymans and Milovic in~\cite[\textsection 4, p. 17]{16-rank} where our $\mathfrak{l}$ corresponds to $\mathfrak{q}$, our integral basis correspond to the generically written basis $\{1,\eta_{1}^{(q)} ,\eta_{2}^{(q)},\eta_{3}^{(q)}\}$ and our units $\varepsilon_q^i$ correspond to the units as $u_i$.


\subsection{Sums of type II}

In this section, we will adapt the proof of~\cite[Proposition 3.8]{16-rank} for our sequence $(a_{\mathfrak{n},q})$ and the field $M_q$, dealing with bilinear sums or sums of type II.

We consider $w$ and $z$ in $\mathcal{O}_{M_q}$ that are coprime with $n_q$. Recalling our definition of the symbol $[\,\cdot\,]$ in~\eqref{[]} and the observation of~\eqref{8}, we have
	\[
	[wz]=\left(\frac{8\tau(wz)\tau\sigma(wz)}{wz}\right)_{M_q,4}\left(\frac{2}{\texttt{u}(wz\sigma(wz))}\right)_{\mathbb{Q},2}.
	\]
We can then rewrite this equality as 
\begin{align*}
	[wz]&=\\
		&[w][z]\,\,Q_2(w,z)\left(\frac{\tau(w)}{z}\right)_{M_q,4}\left(\frac{\tau\sigma(w)}{z}\right)_{M_q,4}\left(\frac{\tau(z)}{w}\right)_{M_q,4}\left(\frac{\tau\sigma(z)}{w}\right)_{M_q,4},
\end{align*}
where 
	\[
	Q_2(w,z):=
		\left(\frac{2}{\texttt{u}(w\sigma(w))}\right)_{\mathbb{Q},2}\left(\frac{2}{\texttt{u}(z\sigma(z))}\right)_{\mathbb{Q},2}\left(\frac{2}{\texttt{u}(wz\sigma(wz))}\right)_{\mathbb{Q},2}.
	\]
We note that $Q_2(w,z)\in \{\pm1,\pm i \}$ depends only on the congruence class of $w$ and $z$ modulo 8.

Now we want to simplify the quartic residue symbols. We use Lemma~\ref{reciprocità quartica} to find some $\mu_1\in \{\pm 1, \pm i\}$ that depends on the congruence classes of $w$ and $z$ modulo 32, such that we have 
\begin{align*}
	\left(\frac{\tau(w)}{z}\right)_{M_q,4}\left(\frac{\tau(z)}{w}\right)_{M_q,4}
		&=\mu_1\left(\frac{z}{\tau(w)}\right)_{M_q,4}\left(\frac{\tau(z)}{w}\right)_{M_q,4}\\
		&=\mu_1\left(\frac{z}{\tau(w)}\right)_{M_q,4}\tau\left(\frac{z}{\tau(w)}\right)_{M_q,4}\\
		&=\mu_1 \left(\frac{z}{\tau(w)}\right)_{M_q,2},
\end{align*}
since $\tau(i)=i.$ For the remaining symbols, we can find some $\mu_2\in \{\pm 1, \pm i\}$ that depends on the congruence classes of $w$ and $z$ modulo 32, such that 
\begin{align*}
	\left(\frac{\tau\sigma(w)}{z}\right)_{M_q,4}\left(\frac{\tau\sigma(z)}{w}\right)_{M_q,4}
		&= \mu_2\left(\frac{z}{\tau\sigma(w)}\right)_{M_q,4}\tau\sigma\left(\frac{z}{\tau\sigma(w)}\right)_{M_q,4}\\
		&=\mu_2\,\, \mathds{1}_{\gcd(z,\tau\sigma(w))=(1)},
\end{align*}
since $\tau\sigma(i)=-i.$ We can then define $\mu_3:=\mu_1	\mu_2 Q_2(w,z)\in \{\pm 1, \pm i\}$ and we get
\begin{equation}
	\label{twisted multiplicativity}
	[wz]=\mu_3 [w][z] \left(\frac{z}{\tau(w)}\right)_{M_q,2}\, \mathds{1}_{\gcd(z,\tau\sigma(w))=(1)}.
\end{equation}
We consider $\{\alpha_\mathfrak{m}\}_\mathfrak{m}$ and $\{\beta_{\mathfrak{n}}\}_\mathfrak{n}$ two bounded sequences of complex numbers. Then
\begin{align*}
	\sum_{\textrm{N}(\mathfrak{m})\leqslant M}
		&\sum_{\textrm{N}(\mathfrak{n})\leqslant N}\alpha_\mathfrak{m}\beta_{\mathfrak{n}}a_{\mathfrak{m}\mathfrak{n}}=\\
		&\frac{1}{12^2}\sum_{w\in \mathcal{D}_3(M)}\sum_{z\in \mathcal{D}_3( N)}\alpha_w\beta_{z}\left(\sum_{i=0}^3s(\varepsilon_3^iwz)[\varepsilon_3^i wz]\frac{1}{2}\left(1+\left(\frac{-3}{\varepsilon_3^i wz}\right)_{M_3,4}\right)\right)
\end{align*}for $q=3 $ and for the other $q\in Q\setminus\{3\}$ we have
\begin{align*}
	\sum_{\textrm{N}(\mathfrak{m})\leqslant M}
		&\sum_{\textrm{N}(\mathfrak{n})\leqslant N}\alpha_\mathfrak{m}\beta_{\mathfrak{n}}a_{\mathfrak{m}\mathfrak{n}}=\\
		&\frac{1}{4^2}\sum_{w\in \mathcal{D}_q(M)}\sum_{z\in \mathcal{D}_q( N)}\alpha_w\beta_{z}\left(\sum_{i=0}^3s(\varepsilon_q^iwz)[\varepsilon_q^i wz]\frac{1}{2}\left(1+\left(\frac{-q}{\varepsilon_q^i wz}\right)_{M_q,4}\right)\right).
\end{align*}
using~\cite[Lemma 3.5]{16-rank} with $F=M_q$ and $n=4$ that tells us that every ideal of $\mathcal{O}_{M_3}$ has twelve different generators in the fundamental domain $\mathcal{D}_3$ and $\mathcal{O}_{M_q}$ has four different generators for $q\in Q\setminus\{3\}$ and defining $\alpha_w:=\alpha_{(w)}$ and $\beta_z:=\beta_{(z)}$.
We note that $s(\varepsilon_q^iwz)$ depends on the congruence class of $wz$ modulo 4 and that $[\varepsilon_q^i wz]=\mu_4 [wz]$ for some $\mu_4\in \{\pm 1,\pm i\}$ depending on the congruence class modulo 32, by Lemma~\ref{reciprocità quartica}.
What is more, the expression $\frac{1}{2}\left(\left(\frac{-q}{\varepsilon_q^i wz}\right)_{M_q,4}+1\right)$ takes values in the set $\{0,1,(1+i)/2,(1-i)/2\}$. This implies that
	\[
	\left|\frac{1}{2}\left(1+\left(\frac{-q}{\varepsilon_q^i wz}\right)_{M_q,4}\right)\right|\leqslant 1.
	\]

We focus on the congruence classes of $w$ and $z$ modulo $q\cdot 2^5$ and so we can bound the previous sums by  a finite number of sums of the form
	\[
	\mu_5\sum_{\substack{w\in \mathcal{D}_q(M)\\w\equiv \omega\bmod q\cdot2^5}}\sum_{\substack{z\in \mathcal{D}_q(N)\\ z\equiv \zeta \bmod q\cdot2^5}}\alpha_w\beta_z[wz],
	\]
where $\mu_5$ depends on the congruence classes $\omega$ and $\zeta$ modulo $q\cdot 2^5.$

We now use our simplification of the symbol $[wz]$ of~\eqref{twisted multiplicativity} and we replace $\alpha_w$ and $\beta_z$ with  $\alpha_w[w]$ and $\beta_z[z]$. Then, if we consider $\mu_6\in\{\pm 1,\pm i\}$ depending only on $\omega$ and $\zeta$, we have
	\[
	\mu_6\sum_{\substack{w\in \mathcal{D}_q(M)\\w\equiv \omega\bmod q\cdot2^5}}\sum_{\substack{z\in \mathcal{D}_q(N)\\ z\equiv \zeta \bmod q\cdot2^5}}\alpha_w\beta_z\left(\frac{z}{\tau(w)}\right)_{M_q,2}\, \mathds{1}_{\gcd(z,\tau\sigma(w))=(1)}.
	\]
The last thing to do is to check that the function 
	\[
	\gamma(w,z):=\left(\frac{z}{\tau(w)}\right)_{M_q,2}\,\mathds{1}_{\gcd(z,\tau\sigma(w))=(1)}
	\]
satisfies the properties (P1), (P2) and (P3) stated in~\cite[Lemma 4.1]{Bilinear sums}.
We can easily see that (P1) follows from Lemma~\ref{reciprocità quartica}, since we are working with congruence classes modulo $q\cdot 2^5$.
Property (P2) is satisfied by the properties of the quadratic residue symbol in $M_q$ given by Proposition~\ref{properties of Hilbert}, Definitions~\ref{n-th power residue symbol}, \ref{def n-th residue 1} and \ref{def n-th residue 2} together with the fact that the indicator function of the $\gcd$ is completely multiplicative and $\mathds{1}_{\gcd(z,\tau\sigma(w))=(1)}=\mathds{1}_{\gcd(w,\tau\sigma(z))=(1)}$.

The first part of property (P3) is given again by the properties of the quadratic residue symbol in $M_q$ and recalling that $\tau(w)$ divides the norm $\textrm{N}_{M_q/\mathbb{Q}}(w).$
For the second part of (P3), we define the function
	\[
	f(w):=\sum_{\xi\bmod \textrm{N}_{M_q/\mathbb{Q}} (w)}\gamma(w,\xi)=\sum_{\xi\bmod \textrm{N}_{M_q/\mathbb{Q}} (w)}\left(\frac{\xi}{\tau(w)}\right)_{M_q,2}\,\mathds{1}_{\gcd(\xi,\tau\sigma(w))=(1)}.
	\]
If $w$ and $w'$ are two elements that generate ideals coprime to $n_q$ and such that $\gcd(\textrm{N}_{M_q/\mathbb{Q}}(w),\textrm{N}_{M_q/\mathbb{Q}}(w'))=~1$, then we have that $f(ww')=f(w)f(w')$. Hence, in order to prove property (P3), we just need to prove that $f(w)=0$ for $w$ that generates a prime ideal coprime to $n_q$ of degree 1. We are sure that we can find such an element that divides a generic $w$, because we have by assumption that $\textrm{N}_{M_q/\mathbb{Q}}(w)$ is not squarefull. 

So let $w$ be an element that generates a prime ideal coprime to $n_q$ of degree 1. Then we have that $w$, $\sigma(w)$, $\tau(w)$, and $\tau\sigma(w)$ are all coprime to each other. 
By the Chinese Remainder Theorem, using these comprimality relations, the function $f(w)$, apart from a non-zero factor, becomes
\begin{multline*}
	\sum_{\xi\bmod\tau(w\sigma (w))}\left(\frac{\xi}{\tau(w)}\right)_{M_q,2}\,\mathds{1}_{\gcd(\xi,\tau\sigma(w))=(1)}=\\
	\sum_{\xi\bmod\tau(w)}\left(\frac{\xi}{\tau(w)}\right)_{M_q,2}\sum_{\xi\bmod\tau\sigma (w)}\,\mathds{1}_{\gcd(\xi,\tau\sigma(w))=(1)}.
\end{multline*}
We note that by~\cite[Lemma 3.6]{spin of prime ideals}, the Dirichlet character given by the quadratic residue symbol is not principal. Hence we obtain the desired result by basic properties of cancellation of Dirichlet characters in a complete set of representatives.

This proves~\cite[Proposition 3.8]{16-rank}. As we saw at the beginning of \textsection 4, we apply~\cite[Proposition 5.2]{spin of prime ideals} to obtain Theorem~\ref{Result}.


\subsection*{Acknowledgements}

This research forms part of my Master Thesis at the University of Leiden with the supervision of P. Koymans and P. Stevenhagen. I am grateful to them for introducing me to the subject and for very useful discussions. Additionally, I wish to thank B. Klopsch for his numerous suggestions which have significantly improved the paper.
I would like to thank the referee for their detailed and thoughtful report, which has improved the exposition of this paper.



\begin{thebibliography}{99}

\bibitem{C-L}
H.~Cohen and H.\,W.~Lenstra Jr., Heuristics on class groups of number fields, in: \textit{Number theory, Noordwijkerhout} 1983, Lecture Notes in
Math. \textbf{1068}, Springer, Berlin (1984), 33--62.

\bibitem{spin of prime ideals}
J.\,B.~Friedlander, H.~Iwaniec, B.~Mazur and K.~Rubin, The spin of prime ideals, \textit{Invent. Math.} \textbf{193} (2013), no. 3, 697--749.

\bibitem{Gerth}
F.~Gerth, Extension of conjectures of Cohen and Lenstra, \textit{Exposition. Math.} \textbf{5} (1987), no. 2, 181--184.

\bibitem{16-rank}
P.~Koymans and D.~Milovic, On the 16-rank of class groups of $\mathbb{Q}(\sqrt{-2p})$ for primes $p\equiv 1 \bmod 4$, \textit{Int. Math. Res. Not. IMRN} (2018), no. 23, 7406--7427.

\bibitem{-p}
P.~Koymans and D.~Milovic, Spins of prime ideals and the negative Pell equation $x^2-2py^2=-1$, \textit{Compos. Math.} \textbf{155} (2019), no. 1, 100--125.

\bibitem{Bilinear sums}
P.~Koymans and D.~Milovic, \emph{Joint distribution of spins}, preprint (2018): \texttt{arXiv:} \texttt{1809.09597}.

\bibitem{Leonard and Williams}
P.\,A.~Leonard and K.\,S.~Williams, On the divisibility of the class number of $\mathbb{Q}(\sqrt{-pq})$ by 16, Number theory (Winnipeg, Man., 1983), \textit{Rocky Mountain J. Math.} \textbf{15} (1985), no. 2, 491.

\bibitem{Milovic}
D.~Milovic, On the 16-rank of class group of $\mathbb{Q}(\sqrt{-8p})$ for $p\equiv -1\bmod 4$, \textit{Geom. Funct. Anal.} \textbf{27} (2017), no. 4, 973--1016.

\bibitem{Neukirch}
J.~Neukirch, \emph{Algebraic number theory}, Translated from the 1992 German original and with a note by Norbert Schappacher, Springer-Verlag, Berlin Heidelberg, 1999.

\bibitem{Smith}
A.~Smith, \emph{$2^\infty$-Selmer groups, $2^\infty$-class groups, and Goldfeld's conjecture},  preprint (2017): \texttt{arXiv:}~\texttt{1708.08509}.

\bibitem{Uchida}
K.~Uchida, Imaginary abelian number fields of degree $2^m$ with class number one, \textit{Proceedings of the international conference on class numbers and fundamental units of algebraic number fields (Katata, 1986)} (1986), 151--170, Nagoya Univ., Nagoya.

\end{thebibliography}
\end{document}